\documentclass{amsart}
\usepackage{amsfonts}
\usepackage{amsmath}
\usepackage{amssymb}
\usepackage{amsthm}
\usepackage[T1]{fontenc}
\usepackage{tensor}
\usepackage{geometry}
\usepackage{mathrsfs}
\usepackage{setspace}
\usepackage{fancyhdr}
\makeatletter
\def\blfootnote{\xdef\@thefnmark{}\@footnotetext}
\makeatother

\newtheorem{thm}{Theorem}[section]

\newtheorem{cor}[thm]{Corollary}

\newtheorem{lem}[thm]{Lemma}

\newtheorem{prop}[thm]{Proposition}
\theoremstyle{definition}
\newtheorem{deff}[thm]{Definition}
\newtheorem{rmk}[thm]{Remark}
\newtheorem{exm}[thm]{Example}
\newcommand{\N}{\mathbb N}

\newcommand{\C}{\mathbb C}
\newcommand{\R}{\mathbb R}

\newcommand{\la}{\langle}
\newcommand{\ra}{\rangle}

\newcommand{\hi}{\mathcal H}

\newcommand{\ki}{\mathcal K}

\begin{document}

\title[Closedness of generators]{The closedness of the generator of a semigroup}

\author[Androulakis]{George Androulakis}
\address{Department of Mathematics\\
University of South Carolina \\ 
Columbia, SC}
\email{giorgis@math.sc.edu}

\author[Ziemke]{Matthew Ziemke}
\address{Department of Mathematics\\
University of South Carolina \\ 
Columbia, SC}
\email{ziemke@email.sc.edu}

\blfootnote{The article is part of the 
second author's Ph.D. thesis which is prepared at the University of South Carolina under the supervision of the first author.}

\maketitle

\begin{abstract}
We study semigroups of bounded operators on a Banach space such that the members of the semigroup are continuous with respect to various weak topologies and we give sufficient conditions for the generator of the semigroup to be closed with respect to the topologies involved.  The proofs of these results use the Laplace transforms of the semigroup.  Thus we first give sufficient conditions for Pettis integrability of vector valued functions with respect to scalar measures.
\end{abstract}

\begin{spacing}{1.3}

\section{Motivation}\label{motivation}
The motivation for this article comes from the area of quantum Markov semigroups (QMSs) and the well known problem of describing the form of the generator of a general QMS.  A semigroup on a Banach space $X$ is a family $(T_t)_{t \geq 0}$ of bounded operators on $X$ satisfying $T_0=1$ and $T_tT_s=T_{t+s}$ for $t,s \geq 0$.

A QMS on a von Neumann algebra $\mathcal{M}$ is a semigroup $(T_t)_{t \geq 0}$ on $\mathcal{M}$ such that each $T_t$ is completely positive, $\sigma$-weakly-$\sigma$-weakly continuous, and satisfies $T_t(1)=1$, and the map $t \mapsto T_tA$ is $\sigma$-weakly continuous for each $A \in \mathcal{M}$.  This definition was finalized in 1976 by Lindblad \cite{lindblad} who was able to describe the form of the generator of the semigroup in the case that the von Neumann algebra $\mathcal{M}$ is semifinite and the map $t \mapsto T_t$ is norm continuous.  Such semigroups are called uniformly continuous.  Lindblad proved that for such QMSs the generator $L$ has the form $L(A)=\phi (A)+G^*A+AG$ for all $A \in \mathcal{M}$, where $G \in \mathcal{M}$ and $\phi$ is a $\sigma$-weakly-$\sigma$-weakly continuous and completely positive map on $\mathcal{M}$.

Then by a result of Stinespring \cite{stinespring}, every completely positive map $\phi$ can be written in the form $\phi (A)=V^* \pi (A)V$, where $\pi : \mathcal{M} \rightarrow \mathcal{B}(\ki)$ is a $^*$-representation of $\mathcal{M}$ into the algebra $\mathcal{B}(\ki)$ of all bounded linear operators on a Hilbert space $\ki$ and $V$ is a bounded linear operator from $\hi$ to $\ki$ (where $\hi$ is the Hilbert space on which $\mathcal{M}$ acts).  Moreover the $\sigma$-weak-$\sigma$-weak continuity of $\phi$ implies $\sigma$-weak-$\sigma$-weak continuity of $\pi$.  Finally by a result of Kraus \cite{kraus} every $\sigma$-weakly-$\sigma$-weakly continuous $^*$-representation $\pi : \mathcal{M} \rightarrow \mathcal{B}(\ki)$ can be written in the form $\pi (A)= \sum_{j=1}^{\infty} V_j^*AV_j$ where $(V_j)_{j=1}^{\infty}$ is a sequence of bounded operators from $\ki$ to $\hi$ and the series $\sum_{j=1}^{\infty} V_j^*AV_j$ converges strongly.

Extensions of the result of Lindblad were made by Davies \cite{davies1, davies2}, Holevo \cite{holevo}, and recently by the authors \cite{az}.  The article \cite{az} studies QMSs on the von Neuman algebra $\mathcal{B}(\hi)$ (where $\hi$ is a Hilbert space) without assuming that the semigroup is uniformly continuous.  An important tool in the work of \cite{az} is the domain algebra $\mathcal{A}$ of the generator $L$ of a semigroup, which was introduced by Arveson \cite{arveson}.  Let $\hi$ be a Hilbert space, $(T_t)_{t \geq 0}$ be a QMS on $\hi$, $L$ be its generator and $D(L)$ be the domain of $L$.  The domain algebra $\mathcal{A}$ is the largest$^*$-subalgebra of $D(L)$ and it's given by
$$\mathcal{A}= \{ A \in D(L): A^*A \text{ and }AA^* \text{ belong to } D(L) \} .$$
Let $e \in \hi$ such that $|e \ra \la e| \in D(L)$, and set
$$U_e= \{ x \in \hi : |x \ra \la e| \in \mathcal{A} \} .$$
For simplicity assume that $\hi=L_2(\R)$, $e$ is a smooth function on $\R$ with compact support, and $U_e$ is the Schwartz class of functions on $\R$ which will hereafter be denoted by $U$ (indeed, it can be easily verified that the Schwartz class is a subspace of $U_e$ in Example \ref{example1} if $\hi=L_2(\R)$, $V_t$ is the translation by $t$ i.e. $V_t(g)(x)=g(x-t)$ for $g \in \hi$, and $e$ is a smooth function with compact support).
The main result of \cite{az} states that there exists a Hilbert space $\ki$, a unital $^*$-representation $\pi : \mathcal{A} \rightarrow \mathcal{B}(\ki)$ and linear maps $G:U \rightarrow \hi$ and $V:U \rightarrow \ki$ such that
\begin{equation}\label{result}
\la x, L(A)y \ra_{\hi} = \la Vx, \pi (A)Vy \ra_{\ki} + \la x , GAy \ra_{\hi} + \la GA^*x,y \ra_{\hi} \quad \text{for all } A \in \mathcal{A} \text{ and } x,y \in U,
\end{equation}
where
$$Gx=L(|x \ra \la e|)e- \frac{1}{2} \la e, L(|e \ra \la e|)e \ra x , \quad \text{for all } x \in U .$$
This result is the analogue of the result of Lindblad \cite{lindblad} without incorporating Kraus' result.  Equation \eqref{result} is an equality of bilinear forms on $U$.  This equation can be turned into an equation of operators as follows:  Assume that
\begin{equation}\label{assumption}
\forall x \in U, \, \exists C_x>0 \text{ such that } | \la x, L( |y \ra \la e|)e \ra | \leq C_x \| y \| \quad \text{for all } y \in U.
\end{equation}
Then notice that

(i) $U$ is a subset of the domain $Dom(G^*)$ of $G^*$, and

(ii) for every $y \in U$ and $A \in \mathcal{A}$ we have that $\pi(A)Vy \in Dom(V^*)$.

\noindent Indeed, if we fix $x \in U$, then for any $y \in U$, we have
\begin{align*}
| \la x, Gy \ra | & = | \la x, L( |y \ra \la e|)e \ra -\frac{1}{2} \la e, L(|e\ra \la e|)e \ra \la x, y \ra  |
\\ & \leq C_x \| y \| + \frac{1}{2} | \la e , L(|e \ra \la e|) e \ra | \| x \| \| y \|
\\ & = \left( C_x + \frac{1}{2} | \la e , L(|e \ra \la e|) e \ra | \| x \| \right) \| y \|.
\end{align*}
Hence, $x \in Dom(G^*)$.  Next, if we fix $A \in \mathcal{A}$ and $x \in U$, then for any $y \in U$,
\begin{align*}
| \la \pi (A)Vx, Vy \ra_{\ki}| & = | \la Vy, \pi (A)Vx \ra_{\ki} |
\\ & \leq | \la y, L(A)x \ra | + | \la y , GAx \ra | + | \la GA^*y,x \ra | .
\end{align*}
Further,
$$| \la y, L(A)x \ra | \leq \| y \| \| L(A)x \| ,$$
also,
$$| \la y , GAx \ra | \leq \| y \| \| GAx \| ,$$
and lastly,
$$ | \la GA^*y,x \ra | = | \la A^*y, G^*x \ra | = | \la y , AG^*x \ra | \leq \| y \| \| AG^*x \| .$$
Hence, $\pi (A)Vx \in Dom(V^*)$.

Statements (i) and (ii) imply that Equation \eqref{result} can be equivalently rewritten as
\begin{equation}\label{result2}
\la x, L(A)y \ra_{\hi} = \la x, \left(V^*\pi (A)V+GA+AG^* \right)y \ra_{\hi} \quad \text{for all } A \in \mathcal{A} \text{ and } x,y \in U.
\end{equation}
If we further assume that
\begin{equation}\label{assumption2}
\text{ the linear subspace } U \text{ is dense in } \hi
\end{equation}
then Equation \eqref{result2} can be equivalently written as 
\begin{equation}\label{result3}
L(A)y= \left( V^* \pi (A)V + GA+AG^* \right)y \text{ for all } A \in \mathcal{A} \text{ and } y \in U.
\end{equation}
Thus, under assumptions \eqref{assumption} and \eqref{assumption2} the equality \eqref{result} of sesquilinear forms becomes equality \eqref{result3} of operators.

The operators $V:U \rightarrow \hi$ and $G:U \rightarrow \ki$ and their adjoints that appear in \eqref{result3} are not necessarily bounded.  Let $\mathcal{L}(U)$ denote the set of all (not necessarily bounded) linear operators on $\hi$ whose domains contain the linear subspace $U$ of $\hi$.  In order to prove that the $^*$-representation $\pi$ which appears in \eqref{result3} is continuous with respect to some topologies one needs to show that the map $\varphi : \mathcal{A} \rightarrow \mathcal{L}(U)$ defined by
\begin{equation}\label{varphidef}
 \varphi (A)=V^*AV, \quad A \in \mathcal{A}
 \end{equation}
is closed with respect to appropriate topologies.  Given a QMS $(T_t)_{t \geq 0}$, the topologies with respect to which each map $T_t$ is continuous, are the same as the topologies used for the closedness of the map $\varphi$ defined in \eqref{varphidef}.  Even though for a general QMS it is assumed that every $T_t$ is $\sigma$-weakly-$\sigma$-weakly continuous, the $\sigma$-weak topology can not be naturally defined in the set $\mathcal{L}(U)$ of generally unbounded operators (which contains the range of the map $\varphi$).  On the other hand, a natural topology on the space $\mathcal{L}(U)$ of linear operators whose domain contains the dense linear subspace $U$, is the {\bf weak operator topology on $U$}, which can be denoted by $WOT(U)$ and it is the locally convex topology generated by the family of seminorms $(p_{x,y})_{x,y \in U}$ where, for $x,y \in U$ and $T \in \mathcal{L}(U)$, we define $p_{x,y}(T)= | \la x, Ty \ra |$.  If $\hi$ is a Hilbert space and $U$ is a linear subspace of $\hi$ then the $WOT(U)$ topology on $\mathcal{B}(\hi)$ is the $\sigma(\mathcal{B}(\hi) , F_U)$ topology where $F_U$ is the linear span of the set $\{ |x \ra \la y|:x,y \in U \}$ which consists of rank-1 operators.  The set $F_U$ is considered as a subset of $\mathcal{B}(\hi)^*$ but since it consists of finite rank operators we can consider it as a subset of the trace class operators $L_1 (\hi)$ which acts on $\mathcal{B}(\hi)$ by trace duality.  It can be easily proved that the assumption \eqref{assumption2} is equivalent to the fact that $F_U$ is dense in $L_1(\hi)$ hence if $\overline{F_U}$ denotes the closure of $F_U$ in $\mathcal{B}(\hi)^*$ then the $\sigma(\mathcal{B}(\hi), \overline{F_U})$ topology is the $\sigma$-weak topology on $\mathcal{B}(\hi)$.

In this article we study semigroups on a Banach space $X$ such that each member of the semigroup is $\sigma(X,F)$-$\sigma(X,F)$ continuous where $F$ is a linear subspace of $X^*$, and we prove that the generator of the semigroup is closed with respect to these topologies.  The main results of the article are Theorems \ref{15} and \ref{19} and Corollary~\ref{20} which give sufficient conditions for the generator of a semigroup to be closed with respect to various weak topologies.  The main tools for proving these results are Theorems \ref{5} and \ref{8} which provide sufficient conditions for a normed space valued function to be Pettis integrable.

\section{Sufficient conditions for Pettis integrability}\label{generalpettis}
In this section we first recall Pettis integrability of normed space valued functions as in \cite{du}.  The main results of the section are Theorems \ref{5} and \ref{8} which give sufficient conditions for a function to be Pettis integrable.  These theorems are similar to results in \cite{kunze}.
\begin{deff}
Let $(\Omega , \Sigma , \mu)$ be a measure space.  Let $X$ be a normed space and let $F$ be a linear subspace of $X^*$.  If $f: \Omega \rightarrow X$ is such that $\eta \circ f$ is $\Sigma$-measurable for all $\eta \in F$ then we say that $f$ is { \bf F-measurable}.  Further, if  $f$ is $F$-measurable and $\eta \circ f \in L_1 ( \Omega)$ for all $\eta \in F$ then we say $f$ is { \bf F-Dunford integrable}.
\end{deff}

\begin{prop}\label{2}
Let $(\Omega , \Sigma , \mu )$ be a measure space.  Let $X$ be a Banach space and $F$ be a norm closed subspace of $X^*$.  If $f: \Omega \rightarrow X$ is $F$-Dunford integrable then for any $E \in \Sigma$, there exists $B_E \in F^*$ such that 
$$B_E ( \eta) = \int_E \eta \circ f (t) d \mu (t), \quad \text{for all } \eta \in F .$$
\end{prop}

\begin{proof}
Let $E \in \Sigma$ and define $T:F \rightarrow L_1(\Omega)$ by $T(\eta)=(\eta \circ f)\chi_E$ for all $\eta \in F$.  We first show that $T$ is closed so let a sequence $(\eta_n)_{n \in \N} \subseteq F$ such that $\eta_n \rightarrow \eta$ for some $\eta \in F$ (with respect to the norm of $F$), and $T(\eta_n) \rightarrow g$ in $L_1(\Omega)$.  Then, there exists a subsequence $(\eta_{n_k})_{k \in \N}$ such that $T(\eta_{n_k}) \rightarrow g$ almost everywhere, hence $(\eta_{n_k} \circ f)(t)\chi_E(t) \rightarrow g(t)$ for almost every $t \in \Omega$.  On the other hand, since $\eta_n \rightarrow \eta$ in norm we have that $(\eta_n \circ f)(t)\chi_E(t) \rightarrow ( \eta \circ f)(t) \chi_E(t)$ for all $t \in \Omega$.  Thus we must have that $T( \eta)=g$ and so $T$ is closed.  Since $F$ is a Banach space, by the closed graph theorem, $T$ is continuous and so
$$ \int_E | \eta \circ f(t)| d \mu (t) = \| T( \eta) \|_1 \leq \| T \| \| \eta \| .$$
Hence the linear map $B_E:F \rightarrow \C$ defined by 
$$B_E( \eta) = \int_E \eta \circ f(t) d \mu (t)$$ is continuous on $F$ and therefore $B_E \in F^*$.

\end{proof}

In the situation described in Proposition \ref{2} we write
$$B_E= (D)-\int_E f(t) d \mu (t) .$$

We would like a similar result to the above in the case when $F$ is a linear subspace of $X^*$ which is not necessarily closed.  This is achieved by replacing the assumption ``f is $F$-Dunford integrable'' by the stronger assumption ``$\| f(t) \| \leq g(t)$ for $\mu$-almost all $t \in \Omega$ for some $g \in L_1 (\Omega)$'' as in the following:

\begin{prop}\label{2.5}
Let $(\Omega , \Sigma , \mu )$ be a measure space.  Let $X$ be a Banach space and $F$ be a linear subspace of $X^*$.  Further, suppose $f: \Omega \rightarrow X$ is $\Sigma$-measurable.  If there exists $g \in L_1(\Omega )$ such that $\| f(t) \| \leq g(t)$ for $\mu$-almost all $t \in \Omega$ then $f$ is $F$-Dunford integrable and for any $E \in \Sigma$, there exists $B_E \in F^*$ such that
$$B_E( \eta ) = \int_E \eta \circ f(t) d\mu(t)  , \quad \text{for all } \eta \in F .$$
\end{prop}

\begin{proof}
Clearly $\eta \circ f \in L_1(\Omega )$ for all $\eta \in F$ so $f$ is $F$-Dunford integrable.  Let $E \in \Sigma$.  Define $B_E:F \rightarrow \C$ by
$$B_E( \eta )= \int_E \eta \circ f(t) d\mu(t) , \quad \text{for all } \eta \in F .$$
Then, for any $\eta \in F$, we have
$$ |B_E( \eta )| \leq \int_E| \eta \circ f(t)| d\mu(t) \leq \| \eta \| \int_E \| f(t) \| d\mu(t) \leq \| \eta \| \int_E g(t) d\mu(t) \leq \| \eta \| \| g \|_1$$
and so $B_E$ is bounded and hence $B_E \in F^*$ which completes the proof.
\end{proof}

\begin{deff}
Let $(\Omega, \Sigma , \mu)$ be a measure space, $X$ a normed space, and $F$ a subspace of $X^*$.  If $f: \Omega \rightarrow X$ is $F$-Dunford integrable and 
$$(D)-\int_E f(t) d \mu (t) \in X$$ for all $E \in \Sigma$ then $f$ is said to be an {\bf F-Pettis integrable function with X-valued integrals} and we write
$$(D)-\int_E f(t) d \mu (t)=(P)-\int_E f(t) d \mu (t).$$
\end{deff} 

 When the measure $\mu$ is understood then we do not have to specify which measure is considered for Dunford and Pettis integrability.  Otherwise, we will specify the measure with respect to which the Dunford or Pettis integrability is concerned. 

At first it may seem redundant that for an $F$-Pettis integrable function $f$ we specify that it has $X$-valued integrals.  This detail is important when we allow $X$ to be non-complete as it can be seen in Theorems \ref{5} and \ref{19} and Corollary \ref{20}.  If $X$ is complete then we do not insist on mentioning the fact that the integrals are in $X$ (since there is no superspace of $X$ to cause confusion about the containment of the integrals).  In this case we simply say ``$f$ is an {\bf $F$-Pettis integrable function}''.

It is worth noting that if a Banach space $X$ has a predual, say $X_*$, and $F=X_*$ then $F$-Dunford integrability automatically implies $F$-Pettis integrability because, by Proposition \ref{2}, the $F$-Dunford integral $B_E \in F^*=X$.  One way to determine if a Banach space $X$ has a predual is the following result.

\begin{prop}
Let $X$ be a Banach space and let $F$ be a linear subspace of $X^*$.  If the $\sigma (F,X)$ closure of every convex subset of $F$ is equal to its norm closure then $F^*=X$.
\end{prop}

\begin{proof}
Let $B \in F^*$.  We want to show that $B$ is $\sigma (F,X)$ continuous.  Suppose this is not the case.  Then there exists a net $(\eta_i)_{i \in I} \subseteq F$ and $\eta \in F$ such that $\eta_i \rightarrow \eta$ in the $\sigma (F,X)$ topology but $B( \eta_i)$ does not converge to $B( \eta)$.  Thus there exists $\epsilon > 0$ so that for every $\alpha \in I$ there exists $\beta \geq \alpha$ such that $B( \eta_{\beta}) \notin O$ where 
$$O= \left( \Re{(B(\eta))} - \epsilon, \Re{(B(\eta))} + \epsilon \right) \times \left( \Im{(B(\eta))} - \epsilon, \Im{(B(\eta))} + \epsilon \right) .$$
So we can define a subnet $(B(\eta_{j}))_{j \in J}$ such that $B( \eta_j) \notin O$ for all $j \in J$.  Since,
$$
O^c  = \{ z : \Re{z} \geq \Re{(B(\eta))} + \epsilon \}
 \cup \{ z : \Re{z} \leq \Re{(B(\eta))} - \epsilon \}
 \cup \{ z : \Im{z} \geq \Im{(B(\eta))} + \epsilon \}
 \cup \{ z : \Im{z} \leq \Im{(B(\eta))} - \epsilon \}
$$
and $\{ j \in J: B( \eta_j) \in O^c \}$ is a final set which can be written as a union of four sets, at least one of these four sets is a final set.  Thus there exists a closed ``half'' plane $H$ of $\C$ such that $J_0= \{ j \in J: B( \eta_j) \in H \}$ is a final set.  Now, $\eta_j \underset{J_0}{\rightarrow} \eta$ in the $\sigma(F,X)$ topology so $\eta \in \overline{co \{\eta_j: j \in J_0 \} }^{\sigma(F,X)}$.  Then, by assumption $\eta \in \overline{co \{\eta_j: j \in J_0 \} }^{\sigma(F,X)}=\overline{co \{\eta_j: j \in J_0 \} }^{\| \cdot \|}$.  Hence, there exists a sequence $(\zeta_k)_{k \geq 1} \subseteq co \{\eta_j: j \in J_0 \}$ such that $\zeta_k \rightarrow \eta$ in norm.  Since $\zeta_k \in co \{\eta_j: j \in J_0 \}$ we have that $\zeta_k= \sum_{j \in J_k} \alpha_{k,j} \eta_j$ for some $J_k \subseteq J_0$ and $\alpha_{k,j} \geq 0$ where $\sum_{j \in J_k} \alpha_{k,j} = 1$.  Then 
$$B( \zeta_k)= \sum_{j \in J_k} \alpha_{k,j}B(\eta_j)$$
and $B( \eta_j) \in H$ for all $j \in J_k$ so $\sum_{j \in J_k} \alpha_{k,j}B(\eta_j) \in H$ as a convex combination of the elements $B(\eta_j)$'s of the convex set $H$.  So $B(\zeta_k) \in H$ for all $k \in \N$ and $B( \zeta_k) \rightarrow B(\eta)$, hence $B(\eta) \in H$ since $H$ is closed which is a contradiction.  So we must have that $B$ is $\sigma (F,X)$ continuous and therefore $F^*=X$.
\end{proof}

\begin{deff} \label{4.5}
Let $X$ be a normed space and let $F$ be a subspace of $X^*$ which separates points in $X$ (thus the $\sigma (X,F)$ topology on $X$ is Hausdorff).  We say the pair $(X,F)$ satisfies the {\bf Krein-Smulian Property} if the $\sigma(X,F)$ closure of the convex hull of any norm-bounded $\sigma(X,F)$ compact subset of $X$ is $\sigma(X,F)$ compact.
\end{deff}
The previous definition is motivated by the classical Krein-Smulian Theorem which states that for any Banach space $X$, the pair $(X,X^*)$ satisfies the Krein-Smulian property.
The next result strengthens Proposition 2.5.18 of \cite{br}.

\begin{thm}\label{5}
Let $X$ be a normed space and let $F$ be a norming linear subspace of $X^*$ such that $(X,F)$ satisfies the Krein-Smulian property.  Suppose $(\Omega , \tau )$ is a topological space, $\Sigma$ is the Borel $\sigma$-algebra generated by $\tau$ and $\mu : \Sigma \rightarrow [0, \infty ]$ is a $\sigma$-compact measure.  Let $\overline{X}$ be the completion of $X$, (hence $\overline{X}$ is a Banach space), and $f: \Omega \rightarrow X \subseteq \overline{X}$ be a function such that $\eta \circ f$ is continuous for all $\eta \in F$.  If there exists $g \in L_1(\Omega, \Sigma , \mu)$,  which is bounded on compact sets, and satisfies $\| f(t) \| \leq g(t)$ for $\mu$-almost all $t \in \Omega$ then $f$ is an $F$-Pettis integrable function with $\overline{X}$-valued integrals.
\end{thm}

\begin{proof}
Let $\eta \in F$.  Since $\eta \circ f$ is continuous, we have that $\eta \circ f$ is $F$-measurable.  Further, 
$$ \int_{\Omega} |\eta \circ f(t)| d \mu (t) \leq \| \eta \| \int_{\Omega} g(t) d \mu (t) < \infty$$
so $f$ is $F$-Dunford integrable.  Therefore, by Proposition \ref{2.5}, for all $E \in \Sigma$ there exists $B_E \in F^*$ such that
$$B_E( \eta ) = \int_E \eta \circ f(t) d \mu (t), \quad \text{for all }  \eta \in F .$$
In order to show $f$ is in fact an $\overline{X}$-valued $F$-Pettis integrable function it suffices to prove $B_E$ is $\sigma (F,\overline{X})$-continuous.  For the proof of this fact we separate two cases.  

In the first case we assume that the measure $\mu$ has compact support $R$.  We need to define the Mackey topology on $F$, which we'll denote by $\tau (F,X)$.  The Mackey topology $\tau (F,X)$ is the weakest topology on $F$ defined by the seminorms
$$ \eta \mapsto \sup_{c \in K} | \eta (c) |$$
for all convex, $\sigma(X,F)$-compact, circled subsets $K$ of $X$.  The Mackey topology $\tau(F,X)$ is stronger than the weak topology $\sigma (F,X)$, and by the Mackey-Arens theorem, (see \cite[pg. 131]{sw}), $\tau(F,X)$ is the strongest topology on $F$ such that the set of continuous functionals on $(F, \tau(F,X))$ is equal to $X$.

Since $\eta \circ f: \Omega \rightarrow \C$ is continuous for every $\eta \in F$ we obtain that if $(U_i)_{i \in I}$ is a cover of $f(R)$ by $\sigma(X,F)$ basic open sets then $(f^{-1}(U_i))_{i \in I}$ is a cover of $R$ by open sets.  Since $R$ is compact we obtain a finite subcover $(f^{-1}(U_i))_{i \in J}$ (for some finite subset $J$ of $I$) of $R$.  Thus $(U_i)_{i \in J}$ is a finite cover of $f(R)$.  Hence $f(R)$ is $\sigma(X,F)$-compact.

It is easy to enlarge the set $f(R)$ and produce a circled $\sigma(X,F)$ compact subset of $X$.  Indeed let $S$ be the unit circle in $\C$ and let $\psi :S \times X \rightarrow X$ defined by $\psi( \gamma , x )= \gamma x$.  Obviously $\psi$ is continuous when $S$ is equipped with its relative topology of the complex numbers and $X$ is equipped with $\sigma (X,F)$.  Thus $\psi (S \times f(R))$ is a circled $\sigma (X,F)$-compact subset of $X$ which contains $f(R)$.

The set $\psi (S \times f(R))$ may not be convex but we can enlarge it to produce a convex, circled, and $\sigma(X,F)$-compact subset of $X$.  Indeed, since $(X,F)$ has the Krein-Smulian property, we have that the convex hull $K_0$ of $\psi(S \times F(R))$ is $\sigma (X,F)$-compact.  Obviously $K_0$ is also circled and convex.

For $\eta \in F$,
$$
|B_E( \eta)|  = \left| \int_E \eta \circ f(t) d \mu (t) \right|
 \leq \int_R | \eta \circ f(t) | d  \mu  (t)
 \leq  \mu  (R) \sup_{t \in R} | \eta \circ f(t) |
 \leq  \mu  (R) \sup_{x \in K_0} | \eta (x) | .
$$
Hence, $B_E$ is $\tau(F,X)$-continuous and thus, by the Mackey-Arens Theorem, $B_E$ is $\sigma (F,X)$-continuous.  Therefore $B_E \in X \subseteq \overline{X}$.  

In the second case suppose that $\mu$ does not have compact support.  Since $\mu$ is $\sigma$-compact, there exists an increasing sequence of compact sets $R_n$ such that $R_n \nearrow \Omega$.  Then, for each $n \in \N$, there exists $B_n \in X$ such that
$$\eta(B_n) = \int_{E \cap R_n} \eta \circ f(t) d \mu (t), \quad \text{for all }  \eta \in F .$$
For $\eta \in F$ we have that
\begin{align*}
\vert \eta (B_n) - B_E( \eta) \vert & = \left| \int_{E \cap R_n} \eta \circ f(t) d \mu (t) - \int_E \eta \circ f(t) d \mu (t) \right| \leq \int_{E\backslash R_n} | \eta \circ f(t)|  \mu  (t)
\\ & \leq \| \eta \| \int_{\Omega \backslash R_n} \| f(t) \| d  \mu  (t) \leq \| \eta \| \int_{\Omega \backslash R_n} g(t) d  \mu  (t) . \stepcounter{equation}\tag{\theequation}\label{eq1}
\end{align*}
Thus, for $n,m \in \N$ with $m<n$,
$$ |\eta(B_n) - \eta(B_m)| \leq 2 \| \eta \| \int_{\Omega \backslash R_m} g(t) d  \mu  (t) .$$
Hence, since $F$ norms $X$,
$$ \| B_n - B_m \| =\sup_{\substack{ \eta \in F \\ \| \eta \| \leq 1}} |\eta(B_n - B_m) | \leq 2 \int_{\Omega \backslash R_m} g(t) d  \mu  (t) \rightarrow 0$$
as $m \rightarrow \infty$ since $g \in L_1 ( \Omega)$.  
Thus, $(B_n)_{n \in \N}$ is Cauchy in $X$ and therefore there exists $B \in \overline{X}$ such that $B_n \rightarrow B$, hence $\eta (B_n) \rightarrow \eta (B)$ for all $\eta \in F$ (obviously $\eta$ extends continuously to $B \in \overline{X}$).
Note also that by \eqref{eq1} and the fact that $g \in L_1( \Omega)$, 
$$ \eta(B_n) \rightarrow B_E( \eta), \quad \text{ for all } \eta \in F.$$
Thus, $B_E( \eta) = \eta (B)$ for all $\eta \in F$ and therefore $B_E \in \overline{X}$.
\end{proof}

In \cite{gs} equivalent conditions are given for the Krein-Smulian property for real Banach spaces.  There are many papers in the literature discussing this property, for example, in \cite{cmv}, it is proven that if $X$ is a real Banach space which does not contain a copy of $\ell_1[0,1]$ and $F$ is a norming subset of the unit ball of $X^*$, then $(X,F)$ has the Krein-Smulian property.

Next, we look at another sufficient condition for a Banach space valued function to be $F$-Pettis integrable.  We begin with the definition of a Mazur space which was introduced in \cite{wilansky}.

\begin{deff}\label{6}
A {\bf Mazur space} is a locally convex topological vector space $(X, \tau)$, where $\tau$ is the topology on the set $X$, such that every sequentially continuous linear functional $f:(X,\tau) \rightarrow \C$ is continuous.
\end{deff}

There are several papers in the literature that study implications of the fact that the dual $X^*$ of a Banach space $X$ is a Mazur space with respect to the weak$^*$ topology to the Pettis integrability of $X$-valued functions.  See for example \cite{edgar}, \cite{huff}, and \cite{stefannson}.  In the next result rather than considering the weak$^*$ topology on $X^*$ we consider the $\sigma(F,X)$ topology on a linear subspace $F$ of $X^*$.

\begin{thm}\label{8}
Let $X$ be a Banach space and let $F$ be a linear subspace of $X^*$ which separates points in $X$ such that $(F, \sigma (F,X))$ is a Mazur space.  Let $(\Omega , \Sigma, \mu)$ be a measure space. If $f:\Omega \rightarrow X$ is $F$-measurable and $\|f(t) \| \leq g(t)$ for $\mu$-almost all $t \in \Omega$ for some $g \in L_1(\Omega)$ then $f$ is an $F$-Pettis integrable function.
\end{thm}

\begin{proof}
Let $E \in \Sigma$.  Since $f$ is $F$-measurable and $\|f(t) \| \leq g(t)$ for $\mu$-almost all $t \in \Omega$ and $g \in L_1(\Omega)$, by Proposition \ref{2.5} there exists $B_E \in F^*$ such that 
$$B_E( \eta) = \int_E \eta \circ f(t) d\mu (t) \quad \text{for all }\eta \in F.$$  Next we verify that $B_E$ is $\sigma(F,X)$ sequentially continuous on $F$.  Let $(\eta_n)_{n \in \N}$ be a sequence in $F$ such that $\eta_n \rightarrow \eta$ in the $\sigma(F,X)$ topology, for some $\eta \in F$.  By the Uniform Boundedness Principle we have that $M= \sup_{n \in \N} \| \eta_n \| < \infty$. Further, $\eta_n (x) \rightarrow \eta(x)$ for all $x \in X$, hence, $\eta_n \circ f(t) \rightarrow \eta \circ f(t)$ for all $t \in \Omega$.  Then we have that $| \eta_n \circ f(t)| \leq Mg(t)$ for $\mu$-almost all $t \in \Omega$ and $g \in L_1(\Omega)$ so, by the Dominated Convergence Theorem, we have that 
$$\int_E\eta_n \circ f(t) d\mu (t) \rightarrow \int_E \eta \circ f(t) d \mu (t)$$
i.e. $B_E( \eta_n) \rightarrow B_E (\eta)$.  Hence, $B_E$ is $\sigma (F,X)$ sequentially continuous on $F$ and since $(F, \sigma(F,X))$ is a Mazur space we have, by \cite[Theorem 1.3.1]{kr} that $B_E \in X$.  Therefore, $f$ is an $F$-Pettis integrable function. 
\end{proof}

Next, we give a result which yields an immediate application for Theorem \ref{8}.

\begin{prop}\label{7}
If $X$ is a separable Banach space and $F$ is a linear subspace of $X^*$ then $(F, \sigma(F,X))$ is a Mazur space.
\end{prop}

\begin{proof}
Suppose $\phi$ is a $\sigma (F,X)$ sequentially continuous linear functional on $F$.  We want to show that $\phi$ is $\sigma (F,X)$ continuous.   Since $X$ is a Banach space, by \cite[Exercise 1.9.14]{kr}, it is enough to show $\phi$ restricted to the unit ball $Ba(F)$ of $F$ is $\sigma(F,X)$ continuous.  By \cite[page 426]{ds}, $Ba(F)$ with the $\sigma (F,X)$ topology is metrizable so sequential $\sigma(F,X)$ continuity and $\sigma (F,X)$ continuity are equivalent on $Ba(F)$.  Thus $\phi$ restricted to $Ba(F)$ is $\sigma(F,X)$ continuous.
\end{proof}

\begin{cor}\label{9}
Let $X$ be a separable Banach space and let $F$ be a norm-closed subspace of $X^*$. If $(\Omega, \Sigma, \mu)$ is a measure space,  $f:\Omega \rightarrow X$ is $F$-measurable, and $\|f(t) \| \leq g(t)$ for $\mu$-almost all $t \in \Omega$ for some $g \in L_1(\Omega)$ then $f$ is an $F$-Pettis integrable function.
\end{cor}

\begin{proof}
Follows immediately from Theorem \ref{8} and Proposition \ref{7}.

\end{proof}

The next result will be used in Section 2 (in the proof of Lemma \ref{14.5}).

\begin{prop}\label{10}
Let $X$ be a normed space, $F$ be a linear subspace of $X^*$ which separates points in $X$ (thus $\sigma (X,F)$ is a Hausdorff topology), $f$ be an $X$-valued $F$-Pettis integrable function, and let $T:X \rightarrow X$ be a $\sigma(X,F)- \sigma (X,F)$ continuous linear operator.  Then $T \circ f$ is an $F$-Pettis integrable function with $X$-valued integrals and
$$T\left( (P)-\int_E f(t) d \mu (t) \right) = (P)-\int_E T(f(t))d \mu (t), \quad \text{for all } E \in \Sigma.$$
\end{prop}

\begin{proof}
Let $E \in \Sigma$.  Since $f$ is an $X$-valued $F$-Pettis integrable function, there exists $B_E \in X$ such that 
$$\eta (B_E) = \int_E \eta (f(t)) d \mu (t), \quad \text{for all } \eta \in F.$$  Since $F$ is a linear subspace of $X^*$ and $T$ is $\sigma (X,F)- \sigma (X,F)$ continuous we have that $\eta \circ T \in F$ for all $\eta \in F$.  Thus  
$$(\eta \circ T)(B_E) = \int_E (\eta \circ T)(f(t)) d\mu (t)$$ i.e.,
$$\eta( T(B_E)) = \int_E \eta ( T(f(t))) d\mu (t),\quad \text{for all }\eta \in F.$$  Hence $T \circ f$ is $F$-Pettis integrable and 
$$T(B_E) = (P)- \int_E T \circ f(t) d\mu (t)$$
i.e.,
$$T\left( (P)-\int_E f(t) d \mu (t) \right) = (P)-\int_E T(f(t))d \mu (t) .$$
\end{proof}

\section{Applications to $\sigma(X,F)$-continuous Semigroups}\label{sectionsemigroup}
In this section we introduce the notion of $\sigma(X,F)$-semigroups and we give sufficient conditions which imply that the generator of a $\sigma(X,F)$-semigroup acting on a Banach space $X$ is $\sigma (X,F) -\sigma(X,F)$ closed.  The main results of the section are Theorems \ref{15} and \ref{19} and Corollary \ref{20}.

\begin{deff}\label{11}
A family $(T_t)_{t \geq 0}$ of bounded linear operators on a normed space $X$ is called a semigroup if it satisfies the following properties:
\\(i)  $T_0=1$, and
\\(ii) $T_{t+s}=T_tT_s$ for all $s,t \geq 0$.

\noindent Let $F$ be a subset of $X^*$ which separates points in $X$ (thus the $\sigma(X,F)$ topology on $X$ is Hausdorff).  We say that a semigroup $(T_t)_{t \geq 0}$ on $X$ is a {\bf $\sigma(X,F)$-semigroup} if
\\(iii) $T_t$ is $\sigma(X,F)- \sigma (X,F)$ continuous for all $t \geq 0$.

\noindent We say a semigroup $(T_t)_{t \geq 0}$ is {\bf $\sigma(X,F)$ continous at zero} if for all $\eta \in F$ and for all $A \in X$ we have that if $t \searrow 0$ then $\eta(T_tA) \rightarrow \eta (A)$.  We say that the semigroup is {\bf $\sigma(X,F)$ continuous} if for all $\eta \in F$ and for all $A \in X$ we have that if $t \rightarrow s$ for some $s \geq 0$ (while $t$ stays non-negative as well) then $\eta(T_tA)\rightarrow \eta(T_sA)$.  A semigroup is said to be {\bf exponentially bounded } if there exists $M \geq 1$ and $\omega \in \R$ such that $\|T_t \| \leq Me^{\omega t}$ for all $t \geq 0$. 
\end{deff}

\begin{rmk}\label{11.5}
Let $X$ be a normed space and $F$ be a linear subspace of $X^*$ which separates points in $X$.  If $(T_t)_{t \geq 0}$ is a $\sigma(X,F)$-semigroup on $X$ then the family of adjoint maps $(T_t^*)_{t \geq 0}$ with $T_t^*:F \rightarrow F$ is defined by $T_t^*( \eta)= \eta \circ T_t$ for all $\eta \in F$.  It is easy to check that $T_t^*$ is $\sigma(F,X)-\sigma(F,X)$ continuous for all $t \geq 0$.  Further, if for any fixed $A \in X$, the map $t \mapsto T_tA$ is $F$-Dunford integrable then for fixed $\eta \in F$, the map $t \mapsto T_t^*(\eta)$ is $X$-Dunford integrable.  Indeed, if the map $t \mapsto T_tA$ is $F$-Dunford integrable for each $A \in X$ then $\eta \circ T_tA \in L_1[0, \infty)$ for all $A \in X$ and $\eta \in F$, that is, $A \circ T_t^* \eta \in L_1[0, \infty)$ for all $A \in X$ and $\eta \in F$ but this is precisely what is needed for  the map $t \mapsto T_t^*(\eta)$ to be $X$-Dunford integrable.
\end{rmk}

\begin{lem}\label{14}
Let $X$ be a Banach space and suppose that $F$ is a linear subspace of $X^*$ which norms $X$. Let $(T_t)_{t \geq 0}$ be a semigroup on $X$ which is $\sigma(X,F)$ continuous at zero.  Then $(T_t)_{t \geq 0}$ is exponentially bounded.
\end{lem}

\begin{proof}
Let $A \in X$.  
\\{\bf Claim}:  There exists $\delta_A > 0$ such that if $C_A= \{ \|T_tA \| : 0 \leq t \leq \delta_A \}$ then $\sup{C_A} < \infty$.
\\  If not, we can find a sequence $(t_n)_{n \in \N} \subset [0, \infty)$ such that $t_n \searrow 0$ but $\| T_{t_n}A \| \rightarrow \infty$.  For $\eta \in F$ we have that $\eta (T_{t_n}A) \rightarrow \eta (A)$ so $\{ | \eta(T_{t_n}A)|: n \in \N \}$ is bounded for all $\eta \in F$ so, by the uniform boundedness principle and the assumption that $F_2$ norms X, $\{ \| T_{t_n}A \|: n \in \N \}$ is bounded contradicting the fact that $\|T_{t_n}A \| \rightarrow \infty$.  So, let $\delta_A>0$ such that $M=\sup{C_A}< \infty$.  

We now claim that for all $A \in X$, $\sup{ \{ \|T_tA \| : 0 \leq t \leq 1 \} } < \infty$.  Indeed, for any $t \in [0,1]$, we have that $t=n \delta_A + \epsilon$ for some integer $n$ such that $0 \leq n \leq \lfloor \delta_A^{-1} \rfloor$ and $0 \leq \epsilon < \delta_A$.  Then
$$ \|T_tA \| = \|T_{\delta_A}^nT_{\epsilon}A \| \leq \|T_{\delta_A} \|^nM \leq K$$
where $K= \max \{M, \|T_{\delta_A} \|^{\lfloor \delta_A^{-1} \rfloor}M \}$ so we have that $\sup \{ \|T_tA \|: t \in [0,1] \} < \infty$. 

By the uniform boundedness principle we obtain that $M=\sup \{ \|T_t \|: t \in [0,1] \} < \infty$.  Let $t \in [0, \infty)$.  Then $t=n+ \epsilon$ for some $n \in \N$ and $0 \leq \epsilon <1$.  Then
$$\|T_t \| = \| T_nT_{\epsilon} \| \leq \| T_1 \|^n \|T_{\epsilon} \| = Me^{\omega n} \leq  Me^{\omega t}$$
where $\omega = \ln{\|T_1 \|}$.  Therefore we have that $(T_t)_{t \geq 0}$ is exponentially bounded.

\end{proof}

\begin{deff}\label{12}
Let $X$ be a normed space, $F$ be a subspace of $X^*$ which separates points in $X$, and $(T_t)_{t \geq 0}$ be a semigroup on $X$ which is $\sigma(X,F)$ continuous at 0.  The {\bf generator} of $(T_t)_{t \geq 0}$ is defined as the linear operator $L$ on $X$, whose domain $D(L)$ consists of those $A \in X$ for which there exists an element $B \in X$ with the property that
$$\eta(B)=\lim_{t \rightarrow 0} \frac{ \eta(T_tA -A)}{t}, \quad \text{for all } \eta \in F.$$  
If $A \in D(L)$ then $L$ is defined by $LA=B$.
\end{deff}

\begin{rmk}\label{remark2}
Let $X$ be a normed space and $F,F_1,F_2$ be subspaces of $X^*$ which separate points in $X$.  Let $(T_t)_{t \geq 0}$ be a semigroup on $X$ which is $\sigma(X,F)$ continuous at 0, let $L$ be its generator, and $D(L)$ be the domain of $L$.  Suppose there exists $\omega \in \R$ such that for every $\lambda \in \C$ with $\Re{\lambda}> \omega$ we have that $\lambda$ belongs to the resolvent of $L$ and moreover the bounded linear operator $(\lambda - L)^{-1}$ is $\sigma(X,F_1)$-$\sigma(X,F_2)$ continuous then $L:D(L) \rightarrow X$ is $\sigma(X,F_2)-\sigma (X,F_1)$ closed.
\end{rmk}

\begin{proof}
In order to show that $L$ is $\sigma(X,F_2)-\sigma(X,F_1)$ closed we fix $\omega \in \R$ as given in the statement and $\lambda \in \C$ with $\Re{\lambda} > \omega$ and we show that $\lambda - L:D(L) \rightarrow X$ is $\sigma(X,F_2)-\sigma(X,F_1)$ closed.  Indeed, suppose $(A_{\gamma})_{\gamma} \subseteq D(L)$ is a net which converges to some element $A \in X$ in the $\sigma(X,F_2)$ topology and $((\lambda - L)(A_{\gamma}))_{\gamma}$ converges in the $\sigma (X,F_1)$ topology to some element $B \in X$.  Since $(\lambda-L)^{-1}$ is $\sigma(X,F_1)-\sigma(X,F_2)$ continuous we have that
$$( \lambda - L)^{-1}(\lambda - L)A_{\gamma} \rightarrow (\lambda - L)^{-1}B$$
 i.e., $A_{\gamma} \rightarrow (\lambda - L)^{-1}B$ in the $\sigma(X,F_2)$ topology, and hence $(\lambda - L)^{-1}B=A$ so $A \in D(L)$ and $(\lambda -L)A=B$.  Therefore $\lambda - L:D(L) \rightarrow X$ is $\sigma(X,F_2)-\sigma(X,F_1)$ closed and so $L:D(L) \rightarrow X$ is $\sigma(X,F_2)-\sigma(X,F_1)$ closed.
\end{proof}

\begin{lem}\label{13}
Let $X$ be a Banach space.  Suppose that $F$ is a norm-closed subspace of $X^*$ which is a norming set for $X$ and assume that $(F, \sigma(F,X))$ and $(X, \sigma(X,F))$ are Mazur spaces.  Further, suppose $(T_t)_{t \geq 0}$ is a $\sigma(X,F)$-semigroup on $X$ which is exponentially bounded and for all $A \in X$, the map $t \mapsto T_tA$ is $F$-measurable with respect to the Borel $\sigma$-algebra.  Then there exists $\omega \in \R$ such that for all $\lambda \in \C$ with $\Re{\lambda}> \omega$ the map $[0, \infty) \ni t \mapsto e^{- \lambda t}T_tA$ is $F$-Pettis integrable (with respect to the Lebesgue measure on $[0, \infty)$) for all $A \in X$ and if $R(\lambda ):X \rightarrow X$ is defined by
\begin{equation}\label{star}
R(\lambda)A=(P)- \int_0^{\infty}e^{- \lambda t}T_tA dt , \quad \text{for all } A \in X
\end{equation}
then $R(\lambda)$ is $\sigma(X,F)-\sigma(X,F)$ continuous.
\end{lem}

\begin{proof}
Let real numbers $M$ and $\omega$ be such that $\| T_t \| \leq Me^{\omega t}$ for all $t \geq 0$.
Fix $\lambda \in \C$ such that $\Re{ \lambda} > \omega$. Since for all $A \in X$, $t \mapsto T_tA$ is $F$-measurable, $(F, \sigma (F,X))$ is a Mazur space, and $ \| e^{-\lambda t}T_tA \| \leq e^{-(\Re{\lambda})t}\| A \| Me^{ \omega t}=Me^{-(\Re{\lambda}-\omega)t} \| A \| \in L_1[0, \infty)$ we obtain, by Theorem \ref{8}, that for all $A \in X$ the function $t \mapsto e^{-(\Re{\lambda})t}T_tA$ is an  $F$-Pettis integrable function (with respect to the Lebesgue measure).  Thus for all $A \in X$, there exists $x_A \in X$ such that
\begin{equation}\label{star2}
\eta(x_A) = \int_0^{\infty} \eta(e^{-\lambda t}T_tA)dt, \quad \text{for all } \eta \in F.
\end{equation}
Define $R(\lambda):X \rightarrow X$ by $R(\lambda)A=x_A$.  In order to show that $R(\lambda)$ is bounded notice that for $A \in X$,
\begin{align*}
\| R(\lambda)A \| & = \|x_A \|= \sup_{\substack{ \eta \in F \\ \| \eta \|=1}} { \left| \eta(x_A) \right| } \quad \text{ since $F$ norms X}
\\ & =  \sup_{\substack{ \eta \in F \\ \| \eta \|=1}} { \left| \int_0^{\infty} \eta (e^{-\lambda t}T_tA) dt \right| } \leq  \sup_{\substack{ \eta \in F \\ \| \eta \|=1}} {  \int_0^{\infty}\left| \eta (T_tA)\right|e^{-(\Re{\lambda}) t} dt } \leq \| A \| \int_0^{\infty} Me^{(\omega - \Re{\lambda})t}dt .
\end{align*}
So $R(\lambda)$ is bounded.  Lastly, in order to show that $R(\lambda)$ is $\sigma (X,F)-\sigma(X,F)$ continuous we need to show that $\eta \circ R(\lambda) \in F$ for all $\eta \in F$.  By \eqref{star2} we have that for $\eta \in F$,
\begin{equation}\label{13.1}
 (\eta \circ R(\lambda))(A) = \int_0^{\infty} \eta (e^{-\lambda t}T_tA) dt = \int_0^{\infty}e^{-\lambda t} T_t^*(\eta)(A) dt.
\end{equation}
 Since $F$ is a Banach space, $(X, \sigma (X,F))$ is a Mazur space, $t \mapsto e^{-\lambda t}T_t^*(\eta)$ is $X$-measurable by Remark~\ref{11.5} and bounded by an $L_1[0, \infty)$ function, we have by Theorem \ref{8} that $t \mapsto e^{-\lambda t}T_t^*\eta$ is an $F$-valued $X$-Pettis integrable function (with respect to the Lebesgue measure on $[0,\infty)$) for all $\eta \in F$.  Thus for all $\eta \in F$ there exists $\psi_{\eta} \in F$ so that 
$$A( \psi_{\eta})= \int_0^{\infty} A( e^{-\lambda t}T_t^*( \eta)) dt, \quad \text{for all } A \in X.$$
Combining with \eqref{13.1} we obtain
$$ \psi_{\eta}(A) = \int_0^{\infty} \eta(e^{-\lambda t}T_tA) dt = (\eta \circ R(\lambda)(A), \quad \text{for all } A \in X$$
and so $ \eta \circ R(\lambda)= \psi_{\eta} \in F$. Thus $R(\lambda)$ is $\sigma(X,F)$-$\sigma(X,F)$ continuous. 
\end{proof}

The conclusion of Lemma \ref{13} states that the Laplace transform of the semigroup is a bounded operator which is continuous with respect to the same topologies as each member of the semigroup family.  This conclusion becomes an assumption of the next lemma.

\begin{lem}\label{14.5}
Let $X$ be a Banach space and let $F$ be a linear subspace of $X^*$ which is norming for $X$.  Assume that $(T_t)_{t \geq 0}$ is a $\sigma(X,F)$-semigroup on $X$ which is $\sigma(X,F)$ continuous at zero such that there exists $\omega \in \R$ such that for all $\lambda \in \C$ with $\Re{\lambda}> \omega$ the map $[0, \infty) \ni t \mapsto e^{- \lambda t}T_tA$ is $F$-Pettis integrable (with respect to the Lebesgue measure on $[0, \infty)$) for all $A \in X$.  Define
$$R(\lambda)A=(P)- \int_0^{\infty}e^{- \lambda t}T_tA dt , \quad \text{for all } A \in X .$$
Then all such $\lambda$ belong to the resolvent of the generator $L$ of the semigroup $(T_t)_{t\geq 0}$ and moreover $R(\lambda)=(\lambda -L)^{-1}$ for all such $\lambda$.
\end{lem}

\begin{proof}
Let $ \eta \in F$, $A \in X$, and $\lambda \in \C$ with $\Re{\lambda} > \omega$.  Since $R(\lambda)A=(P)- \int_0^{\infty}e^{-\lambda t}T_tAdt$ for all $A \in X$ we have that
$$\eta (R(\lambda)A)= \int_0^{\infty}e^{- \lambda t}\eta (T_tA) dt, \quad \text{for all }A \in X \text{ and } \eta \in F.$$
Since $T_t$ is $\sigma (X, F)$-$\sigma (X, F)$ continuous and $F$ is a linear subspace of $X^*$ which is norming for $X$, by Proposition \ref{10}, we have that $R( \lambda)T_tA=T_tR( \lambda)A$.
Thus,
\begin{align*}
\eta( t^{-1}(T_t-I)R(\lambda)A) & = \eta ( t^{-1}R(\lambda)(T_tA)- t^{-1}R(\lambda)A) \quad \text{since $R( \lambda)T_tA=T_tR( \lambda)A$}
\\ & = \frac{1}{t} \int_0^{\infty} e^{- \lambda s} \eta(T_sT_tA)ds - \frac{1}{t} \int_0^{\infty}e^{-\lambda s} \eta (T_s A) ds
\\ & = \frac{1}{t} \int_t^{\infty} e^{- \lambda (u-t)} \eta (T_uA)du - \frac{1}{t} \int_0^{\infty} e^{- \lambda s} \eta (T_sA)ds
\\ & = \frac{1}{t} \int_0^{\infty} ( e^{- \lambda (u-t)} - e^{- \lambda u}) \eta (T_uA) du - \frac{1}{t} \int_0^t e^{- \lambda (u-t)} \eta (T_uA) du .
\end{align*}
Further,
$$ \frac{1}{t} \int_0^{\infty} ( e^{- \lambda (u-t)} - e^{- \lambda u}) \eta (T_uA) du \rightarrow \lambda \eta(R(\lambda)A)$$
and, since $(T_t)_{t \geq 0}$ is $\sigma(X,F)$ continuous at 0,
$$\frac{1}{t} \int_0^t e^{- \lambda (u-t)} \eta (T_uA) du \rightarrow \eta (A) .$$
So, if $L$ denotes the generator of the semigroup $(T_t)_{t \geq 0}$ and $D(L)$ denotes its domain, then $R(\lambda)A \in \mathcal{D}(L)$ for all $A \in X$ and $\eta(( \lambda -L)R(\lambda)A)=\eta(A)$ for all $A \in X$ and all $\eta \in F$.  Since $F$ is norming for $X$ we obtain that $(\lambda -L)R(\lambda )A=A$ for all $A \in X$.  Since $T_t$ and $R( \lambda)$ commute, we obtain similarly $R(\lambda)(\lambda - L)A=A$ for all $A \in D(L)$.  Thus $\lambda$ belongs to the resolvent of $L$ and $R(\lambda)=(\lambda - L)^{-1}$.
\end{proof}

\begin{thm}\label{15}
Let $X$ be a Banach space and let $F$ be a norm closed subspace of $X^*$ which is a norming set for $X$.  Assume that $(F, \sigma(F,X))$ as well as $(X, \sigma(X,F))$ are Mazur spaces. Let $(T_t)_{t \geq 0}$ be a $\sigma(X,F)$-semigroup which is $\sigma(X,F)$ continuous at zero. Then the generator, $L$, of $(T_t)_{t \geq 0}$ is $\sigma(X,F)-\sigma(X,F)$ closed.
\end{thm}

\begin{proof}
By Lemma \ref{14}, $(T_t)_{t \geq 0}$ is exponentially bounded.

Further, fix $A \in X$ and $\eta \in F$ and notice that if $t \searrow s$ we have
$$ \left| \eta(T_tA)- \eta (T_s(A)) \right| = \left| \eta(T_{t-s}T_sA-T_sA) \right| \rightarrow 0$$
since $(T_t)_{t \geq 0}$ is  $\sigma(X,F)$ continuous at zero.  So $t \mapsto \eta (T_tA)$ is right-continuous and so Borel measurable.  Hence, by Lemma \ref{13}, there exists $\omega \in \R$ such that for all $\lambda \in \C$ with $\Re{\lambda}> \omega$ there exists a bounded linear operator $R(\lambda):X \rightarrow X$ which is $\sigma(X,F)-\sigma(X,F)$ continuous and satisfies Equation \eqref{star}.  Therefore by Lemma \ref{14.5} we obtain that $R(\lambda)=(\lambda - L)^{-1}$ for all such $\lambda$.  Thus all $\lambda \in \C$ with $\Re{\lambda}> \omega$ belong to the resolvent of $L$ and for all such $\lambda$ we have that $(\lambda -L)^{-1}$ is $\sigma(X,F)-\sigma(X,F)$ continuous.  Hence by Remark \ref{remark2} we obtain that $L$ is $\sigma(X,F)-\sigma(X,F)$ closed.

\end{proof}

We also have an analogous result when we replace the property of Mazur with the Krein-Smulian property.  First, we will need a lemma.

\begin{lem}\label{18.5}
Let $X$ be a Banach space and $F$ be a linear (not necessarily closed) subspace of $X^*$ which is norming for $X$.  Assume that $(X,F)$ and $(F,X)$ satisfy the Krein-Smulian property.  Let $(T_t)_{t \geq 0}$ be an $\sigma(X,F)$-semigroup on $X$ which is $\sigma(X,F)$ continuous.  Then there exists $\omega \in \R$ such that for all $\lambda \in \C$ with $\Re{\lambda}> \omega$ the map $[0, \infty) \ni t \mapsto e^{- \lambda t}T_tA$ is $F$-Pettis integrable for all $A \in X$ and if we define
$$R(\lambda)A=(P)- \int_0^{\infty}e^{- \lambda t}T_tA dt , \quad \text{for all } A\in X$$
then $R(\lambda)$ is $\sigma(X, \overline{F})$-$\sigma(X,F)$ continuous for all such $\lambda$ where $\overline{F}$ is the closure of $F$ inside $X^*$.
\end{lem}

\begin{proof}
First, by Lemma \ref{14}, we obtain that $(T_t)_{t \geq 0}$ is exponentially bounded.  Thus there exists $M \geq 1$ and $\omega \in \R$ such that $\| T_t \| \leq Me^{\omega t}$ for all $t \geq 0$.  Then we fix $\lambda \in \C$ with $\Re{\lambda} > \omega$ and $A \in X$.  We claim that the function $[0, \infty) \ni t \mapsto e^{- \lambda t}T_tA$ is an $X$-valued $F$-Pettis integrable function with respect to Lebesgue measure.  Indeed we have that $\| e^{- \lambda t} T_tA \| \leq \| A \| Me^{(\Re{\lambda} - \omega ) t} \in L_1([0, \infty), dt)$.
Also, since $t \mapsto T_tA$ is $\sigma(X,F)$ continuous, it is $F$-measurable with respect to the Borel $\sigma$-algebra of $[0, \infty)$.  Finally $(X,F)$ satisfies the Krein-Smulian property.  Thus by Theorem \ref{5}, $t \mapsto e^{-\lambda t} T_tA$ is an $F$-Pettis integrable function with respect to the Lebesgue measure.  So, there exists $x_A \in X$ such that
$$ \eta(x_A) = \int_0^{\infty}e^{- \lambda t} \eta(T_tA)dt, \quad \text{ for all } \eta \in F.$$
Define $R(\lambda):X \rightarrow X$ by $R(\lambda)A=x_A$.  

Next notice that $R( \lambda ):X \rightarrow X$ is bounded since for $A \in X$ we have
$$
\| R(\lambda )A \|  = \|x_A \|
 = \sup_{\substack{ \eta \in F \\ \| \eta \|=1}} { \left| \eta(x_A) \right| } 
 =  \sup_{\substack{ \eta \in F \\ \| \eta \|=1}} { \left| \int_0^{\infty}e^{- \lambda t} \eta (T_tA) dt \right| }
 \leq \| A \| \int_0^{\infty} Me^{(\omega - \Re{\lambda})t}dt .
$$

Next, we show that $R(\lambda)$ is $\sigma (X,\overline{F})-\sigma(X,F)$ continuous where $\overline{F}$ is the closure of $F$ inside $X^*$.  Fix $\eta \in F$.  We need to show that $\eta \circ R(\lambda) \in \overline{F}$.  First, notice that
\begin{equation}\label{19.1}
 (\eta \circ R(\lambda))(A) = \int_0^{\infty}e^{-\lambda t} \eta (T_tA) dt = \int_0^{\infty}e^{-\lambda t} T_t^*(\eta)(A) dt, \quad \text{ for all } A \in X.
\end{equation}
Define $f:[0, \infty) \rightarrow F$ by $f(t)=e^{-\lambda t} T_t^*(\eta)$.  Observe that $f$ takes values in $F$ since $F$ is a linear subspace of $X^*$ and $(T_t)_{t \geq 0}$ is a $\sigma(X,F)$-semigroup.  Since $(F,X)$ has the Krein-Smulian property we obtain by Theorem \ref{5} that $f$ is an $X$-Pettis integrable function with integrals in $\overline{F}$ i.e. there exists $\gamma \in \overline{F}$ so that
$$A( \gamma )= \int_0^{\infty} e^{- \lambda t}A( \eta \circ T_t) dt, \quad \text{ for all } A \in X.$$
That is,
$$\gamma (A)= \int_0^{\infty}e^{- \lambda t} \eta (T_tA)dt= \eta \circ R(\lambda )(A), \quad \text{ for all } A \in X.$$
Hence $\eta \circ R(\lambda)= \gamma \in \overline{F}$ and therefore $R(\lambda)$ is $\sigma (X,\overline{F})-\sigma(X,F)$ continuous.
\end{proof}

The following result strengthens Proposition 3.1.4 of \cite{br}. 

\begin{thm}\label{19}
Let $X$ be a Banach space and $F$ be a linear (not necessarily closed) subspace of $X^*$ which is norming for $X$.  Assume that $(X,F)$ and $(F,X)$ satisfy the Krein-Smulian property.  Let $(T_t)_{t \geq 0}$ be an $\sigma(X,F)$-semigroup on $X$ which is $\sigma(X,F)$-continuous.  Then the generator $L:D(L) \rightarrow X$ of the semigroup  $(T_t)_{t \geq 0}$ is $\sigma(X,F)-\sigma(X,\overline{F})$ closed where $\overline{F}$ is the closure of $F$ inside $X^*$.
\end{thm}

\begin{proof}
First, by Lemma \ref{14}, we obtain that $(T_t)_{t \geq 0}$ is exponentially bounded.  Thus there exists $M \geq 1$ and $\omega \in \R$ such that $\| T_t \| \leq Me^{\omega t}$ for all $t \geq 0$.
By Lemma \ref{18.5}, the map $[0,\infty) \ni t \mapsto e^{-\lambda t}T_tA$ is $F$-Pettis integrable for all $A \in X$ and if we define $$R(\lambda)A=(P)- \int_0^{\infty}e^{- \lambda t}T_tA dt , \quad \text{for all } A \in X$$
then $R(\lambda)$ is $\sigma(X, \overline{F})$-$\sigma(X,F)$ continuous.
By Lemma \ref{14.5} we obtain that $R( \lambda )=( \lambda - L)^{-1}$.  Thus all $\lambda \in \C$ with $\Re{\lambda}> \omega$ belong to the resolvent of $L$ and for all such $\lambda$ we have that $(\lambda - L)^{-1}$ is $\sigma(X,\overline{F})-\sigma(X,F)$ continuous.  Hence by Remark \ref{remark2} we obtain that $L:D(L) \rightarrow X$, is $\sigma (X, F) -\sigma (X,\overline{F})$ closed.
\end{proof}

We can apply the previous theorem to a quantum Markov semigroup $(T_t)_{t \geq 0}$ where each $T_t$ is WOT-WOT continuous, that is, $\sigma(\mathcal{B}(\hi),F)$ continuous where $F$ is the space of finite rank operators on $\hi$.  In doing so we obtain the following result.

\begin{cor}\label{20}
Let $\hi$ be a Hilbert space and $(T_t)_{t\geq 0}$ be a WOT-semigroup on $\mathcal{B}(\hi)$ which is WOT continuous at $0$.  Then its generator is WOT-$\sigma$-weakly closed.  In particular, the generator of any quantum Markov WOT-semigroup is WOT-$\sigma$-weakly closed.
\end{cor}

\begin{proof}
Let $X$ be $\mathcal{B}(\hi)$ and $F$ be the space of finite rank operators as a linear subspace of $\mathcal{B}(\hi)^*$.  Then $(X,F)$ satisfies the Krein-Smulian Property since the $\sigma(X,F)$ topology is the WOT topology on $\mathcal{B}(\hi)$ and the unit ball of $\mathcal{B}(\hi)$ is WOT compact.  Also $(F,X)$ satisfies the Krein-Smulian Property by Alaoglu's Theorem.  Thus we can simply apply Theorem \ref{19} to obtain that the generator $L:D(L) \rightarrow \mathcal{B}(\hi)$, is $\sigma (\mathcal{B}(\hi), F) -\sigma (\mathcal{B}(\hi),\overline{F})$ closed.  This gives the desired result since the $\sigma (\mathcal{B}(\hi), F)$ topology is the WOT topology and the $\sigma (\mathcal{B}(\hi),\overline{F})$ topology is the $\sigma$-weak topology (since the closure of the linear space of finite rank operators inside $\mathcal{B}(\hi)^*$ is equal to the closure of finite rank operators inside the predual of $\mathcal{B}(\hi)$).
\end{proof}

A similar result to Corollary \ref{20} can be stated when the WOT topology is replaced by the WOT(U) topology for some dense subspace $U$ of $\hi$.

The assumption that the quantum Markov semigroup be a WOT-semigroup is true for many examples.  We will look at two of them now.  The first is a basic example which gives a way to produce a QMS from a classical semigroup.

\begin{exm}\label{example1}
Let $(V_t)_{t \geq 0}$ be a strongly continuous contraction semigroup on a Hilbert space $\hi$.  Define $T_t: \mathcal{B}(\hi) \rightarrow \mathcal{B}(\hi)$ by 
$$T_tA=V_tAV_t^* \quad , \quad \text{for all } A \in \mathcal{B}(\hi) .$$
It is easy to see that $(T_t)_{t \geq 0}$ is a QMS.  To show that for a fixed $t \geq 0$, $T_t$ is WOT-WOT continuous, let $(A_{\lambda})_{\lambda} \subseteq \mathcal{B}(\hi)$ such that $A_{\lambda}\xrightarrow{WOT} A$ for some $A \in \mathcal{B}(\hi)$.  Then, for any $x,y \in \hi$,
$$\la x,T_tA_{\lambda}y \ra = \la x , V_tA_{\lambda}V_t^*y \ra = \la V_t^*x, A_{\lambda}V_t^*y \ra \rightarrow \la V_t^*x,AV_t^*y \ra = \la x, T_tA y \ra .$$
Therefore $(T_t)_{t \geq 0}$ is a WOT-semigroup.
\end{exm}

The next example is given in \cite{arveson} and similar examples where produced in \cite{fagnola2} and \cite{powers}.

\begin{exm}\label{4.3}
Let $\hi = L_2[0, \infty)$ and define $V_t: \hi \rightarrow \hi$ by 

\[
(V_tg)(x)=\left\{ \begin{array}{lr}
g(x-t) & \text{if }x \geq t \\
0 & \text{otherwise}
\end{array}
\right.
\]

Then $(V_t)_{t \geq0}$ is a strongly continuous semigroup of isometries whose generator is the differentiation operator.  Let $f \in L_2[0, \infty)$ be what we get by normalizing $u(x)=e^{-x}$ (i.e. $f=\frac{u}{\| u \|}$).  Then define $\omega : \mathcal{B}(\hi) \rightarrow \C$ by $\omega (A)= \la f,Af \ra$.  Define the completely positive maps $T_t: \mathcal{B}(\hi) \rightarrow \mathcal{B}(\hi)$ where
$$T_t(A)= \omega(A)E_t+V_tAV_t^*$$ 
for all $t \geq 0$ where $E_t$ is the projection onto the subspace $L_2[0, t) \subseteq L_2[0, \infty)$.  Then $( T_t)_{t \geq 0}$ is a QMS.  Let $t \geq 0$.  To see that $T_t$ is WOT-WOT continuous, let $(A_{\lambda})_{\lambda} \subseteq \mathcal{B}(\hi)$ such that $A_{\lambda}\xrightarrow{WOT} A$ for some $A \in \mathcal{B}(\hi)$.  Then, for any $h,g \in \hi$,
\begin{align*}
\la h, T_t (A_{\lambda})g \ra & = \la h , \omega (A_{\lambda})E_t g \ra + \la h, V_tA_{\lambda}V_t^*g \ra
\\ & = \la f , A_{\lambda}f \ra \la h, E_t g \ra + \la V_t^*h,A_{\lambda}V_t^*g \ra
\\ & \rightarrow \la f , Af \ra \la h, E_t g \ra + \la V_t^*h,AV_t^*g \ra
\\ &= \la h, T_t(A)g \ra .
\end{align*}
Therefore $T_t$ is WOT-WOT continuous.
\end{exm}

\end{spacing}


\begin{thebibliography}{}

\bibitem{az} Androulakis, G. and Ziemke, M.:
{Generators of Quantum Markov Semigroups}.
 Submitted.  The article can be viewed at http://people.math.sc.edu/androula/research.html

\bibitem{arveson} Arveson, W.:
{The Domain Algebra of a CP-Semigroup}.
Pac. J. Math. \textbf{203(1)} , 67-77 (2002)

\bibitem{br} Bratteli, O. and Robinson, D.:
{Operator Algebras and Quantum Statistical Mechanics I}.
 Springer-Verlag, New York (1979)

\bibitem{cmv} Cascales, B., Manjabacas, G. and Vera, G.:
{A Krein-Smulian Type Result in Banach Spaces}.
Q. J. Math. \textbf{48(2)}, 161-167 (1997)

\bibitem{davies1} Davies, E.B.:
{Quantum Dynamical Semigroups and the Neutron Diffusion Equation}.
Rep. Math. Phys. \textbf{11(2)}169-188 (1977)

\bibitem{davies2} Davies, E.B.:
{Generators of Dynamical Semigroups}
J. Funct. Anal. \textbf{34}, 421-432 (1979)

\bibitem{du} Diestel, J. and Uhl, J. J.:
{Vector Measures}.
American Mathematical Society, Providence (1977)

\bibitem{ds} Dunford, N. and Schwartz, J.:
{Linear Operators- Part I: General Theory}.
 Interscience Publishers, Inc., New York (1967)

\bibitem{edgar} Edgar, G.:
{Measurability in Banach Space II}.
Indiana Univ. Math. J. \textbf{28}, 559-579 (1979)

\bibitem{fagnola2} Fagnola, F.:
{A Simple Singular Quantum Markov Semigroup}.
In: Rebolledo, R. (eds.)
{Stochastic Analysis and Mathematical Physics}, pp. 73-87.
Birkhauser, Boston (2000)

\bibitem{gs} Granero, A.S. and Sanchez, M.:
{The Class of Universally Krein-Smulian Banach Spaces}.
Bull. Lond. Math. Soc. \textbf{39(4)}, 529-540 (2007)

\bibitem{holevo} Holevo, A. S.:
{On the Structure of Covariant Dynamical Semigroups}.
J. Funct. Anal. \textbf{131}, 255-278 (1995)

\bibitem{huff} Huff, R.:
{Remarks on Pettis Integrability}.
Proc. Amer. Math. Soc. \textbf{96(3)}, 402-404 (1986)

\bibitem{kr} Kadison, R. and Ringrose, J.:
{Fundamentals of the Theory of Operator Algebras; Volume I- Elementary Theory}.
Academic Press, London (1983)

\bibitem{kraus} Kraus, K.:
{General State Changes in Quantum Theory}.
Ann. Phys. \textbf{64},311-335 (1970)

\bibitem{kunze} Kunze, M.:
{A Pettis-Type Integral with Applications to Transition Semigroups}.
Czechoslovak Math. J. \textbf{61(2)}, 267-288 (2011)

\bibitem{lindblad} Lindblad, G.:
{On the Generators of Quantum Dynamical Semigroups}.
Commun. math. Phys. \textbf{48},119-130 (1976)

\bibitem{powers} Powers, R. T.:
{New Examples of Continuous Spatial Semigroups of Endomorphisms of $\mathcal{B}(\hi)$}.
Int. J. Math. \textbf{10},215-288 (1999)

\bibitem{sw}  Schaefer, H.H., in assistance with Wolfe, M.:
{Topological Vector Spaces}.
Springer-Verlag, New York (1999)

\bibitem{stefannson} Stef{\'a}nsson, G.F.:
{Pettis Integrability}
Trans. Amer. Math. Soc. \textbf{330(1)}, 401-418 (1992)

\bibitem{stinespring} Stinespring, W. F.:
{Positive Functions on $C^*$-algebras}.
Proc. Amer. Math. Soc., 211-216 (1955)

\bibitem{wilansky} Wilansky, A.:
{Topics in Functional Analysis}.
  Notes by W.D. Laverell.  Lecture Notes in Mathematics, No. 45,
  Springer-Verlag, Berlin-New York (1967)

\end{thebibliography}
\end{document}